\documentclass[12pt]{amsart}
\usepackage{lmodern} 
\usepackage[utf8]{inputenc} 
\usepackage[T1]{fontenc}
\usepackage{amsfonts}
\usepackage{amsmath}
\usepackage{amssymb}
\usepackage{amsthm}
\usepackage[a4paper, left=3cm, right=3cm, top=3cm,
bottom=3cm,dvips]{geometry}
\usepackage{graphicx}
\usepackage[font=small,labelfont=bf]{caption}
\usepackage{color}
\usepackage[normalem]{ulem}

\newtheorem{theorem}{Theorem}[section]
\newtheorem{lemma}[theorem]{Lemma}
\newtheorem{definition}[theorem]{Definition}

\newtheorem{corollary}[theorem]{Corollary}

\newtheorem{remark}[theorem]{Remark}

\newcommand{\N}{{\mathbb N}}

\newcommand{\strange}[1]{\mathrm{\tilde{#1}}}

\title[The polynomial method for list-colouring extendability]{The polynomial method for list-colouring extendability of outerplanar graphs}
\author{Przemys\l{}aw Gordinowicz}
\address{Institute of Mathematics, Lodz University of Technology, \L{}\'od\'z, Poland}
\email{pgordin@p.lodz.pl}

\author{Pawe\l{} Twardowski}
\address{Institute of Mathematics, Lodz University of Technology, \L{}\'od\'z, Poland}
\email{173450@p.lodz.pl}

\begin{document}

\begin{abstract}
We restate theorems of Hutchinson~\cite{Hutchinson} on list-colouring extendability for outerplanar graphs in terms of non-vanishing monomials in a graph polynomial, which yields an Alon-Tarsi equivalent for her work. This allows to simplify her proofs as well as obtain more general results.
\end{abstract}

\vspace{150pt}

\maketitle

\section{Introduction}
In his famous paper~\cite{Thomassen} Thomassen proved that every planar graph is 5-choosable. Actually, to proceed with an inductive argument, he proved the following stronger result.
\begin{theorem}[\cite{Thomassen}]\label{thm:thomassen}
Let $G$ be any plane near-triangulation (every face except the outer one is a triangle) with outer cycle $C$. Let $x$, $y$ be two consecutive vertices on $C$. Then $G$ can be coloured from any list of colours such that the length of lists assigned to $x$, $y$, any other vertex on $C$ and any inner vertex is 1, 2, 3, and 5, respectively. 
\end{theorem}
In other words vertices $x$ and $y$ can be precoloured in different colours. 
Basically, this theorem implies that any outerplanar graph is 3-choosable. Moreover, lists of any two neighbouring vertices can have a deficiency. To formalise this fact we say that a triple $(G, x, y)$, where $G$ is outerplanar graph, $x, y \in V(G)$ are neighbouring vertices is \emph{$(1,2)$-extendable} in the sense that $G$ is colourable from any lists whose length is 1, 2 and 3 for vertex $x$, $y$ and any other vertex, respectively.  

Hutchinson~\cite{Hutchinson} analysed extendability of outerplanar graphs, in the case when the selected vertices are not adjacent, showing that for any two vertices $x, y$ of outerplanar graph $G$ a triple $(G, x, y)$ is $(2, 2)$-extendable. Of course, it is enough to prove this for outerplane 2-connected near-triangulation only, as each outerplane graph can be extended to such a graph just by adding some edges. The main theorem was the following.
\begin{theorem}[\cite{Hutchinson}]\label{thm:hutchinson}
Let $G$ be outerplane 2-connected near-triangulation and $x, y \in V(G)$, $x \neq y$. Let $C \colon V(G) \to \{1,2,3\}$ be any proper 3-colouring of $G$. Then
\begin{enumerate}
\item $(G, x, y)$ is not $(1, 1)$-extendable;
\item $(G, x, y)$ is $(1, 2)$-extendable if and only if $C(x) \neq C(y)$;
\item $(G, x, y)$ is $(2, 2)$-extendable.
\end{enumerate}
\end{theorem} 
Indeed, it is enough to prove the above theorem for near-triangulations with exactly 2 vertices of degree 2 and to let $x$ and $y$ be these degree 2 vertices. Hutchinson called such configurations \emph{fundamental subgraphs}. Such a configuration can be obtained by successively shrinking the outerplane near-triangulation along some chord (inner edge) that separates the component of the graph not containing vertices $x$ and $y$ (in case when $xy \in E(G)$ this reduces to an edge $xy$). The general result follows now by succesive colouring of shrank parts using Theorem~\ref{thm:thomassen} --- the chord is an outer edge of the shrank component and its endpoins (already coloured) are these 2 precoloured vertices. The details are in~\cite{Hutchinson}. Also in~\cite{Hutchinson}, Hutchinson provided further results about extendability of general outerplanar graphs, for which the conditions are more relaxed than those of Theorem~\ref{thm:hutchinson}, allowing for $(1,1)$-extendability.

One important thing is that the proper 3-colouring $C$ mentioned in the theorem above is not in any way connected to possible list colouring of $G$, but is rather an inherent property of the graph. This is due to the fact that every 2-connected outerplane near triangulation has an unique (up to permutation) 3-colouring, i.e the vertices graph can be uniquely partitioned into 3 groups so that in every proper 3-colouring of the graph the vertices in the same group will always have the same colour (the groups in this partition are called \emph{colour classes}, as the partition defines an equivalence relation). The reason for this is that the graph consists entirely of triangles, and every vertex of a given triangle needs to be of different colour.
	
The situation of particular importance is when two vertices are in the same colour class. This can be forced in two ways. One, mentioned in~\cite{Hutchinson}, is the so called \emph{chain of diamonds}, where the diamond is understood as $K_4$ minus an edge. It is obviously a 2-connected outerplane near triangulation, and the two non-neighbouring vertices are always of the same colour. Therefore is we link diamonds together glueing them by the vertices of degree 2, each of the linking vertices will have the same colour. The second way is to attach a diamond to diamond along the common edge (cf.~\cite{PoThIII}). Both of those ways can be seen on Figure~\ref{fig:colouring}.

Recently, Zhu~\cite{Zhu} strengthened the theorem of Thomassen in the language of graph polynomials showing that Alon-Tarsi number of any planar graph $G$ satisfies $AT(G) \le 5$. His approach utilizes a certain polynomial arising directly from the structure of the graph. This \textit{graph polynomial} is defined as:
$$P(G) = \prod_{uv \in E(G), u<v}(u-v),$$
where the relation $<$ fixes an arbitrary orientation of graph $G$. Here we understand $u$ and $v$ both as the vertices of $G$ and variables of $P(G)$, depending on the context. Notice that the  orientation affects the sign of the polynomial only. Therefore individual monomials and the powers of the variables in each monomials are orientation-invariant. We refer the reader to~\cite{combnull, AlonTarsi, Schautz} for the connection between list colourings and graph polynomials.
The approach of Zhu may be described in the following form, analogous to Theorem~\ref{thm:thomassen}. 
\begin{theorem}[\cite{Zhu}]\label{thm:zhu}
Let $G$ be any plane near-triangulation, let $e = xy$ be a boundary edge of $G$. Denote other boundary vertices by $v_1, \dots, v_k$ and inner vertices by $u_1, \dots, u_m$. Then the graph polynomial of $G-e$ contains a non-vanishing monomial of the form $\eta x^0y^0v_1^{\alpha_1}\dots v_k^{\alpha_k} u_1^{\beta_1}\dots u_m^{\beta_m}$ with $\alpha_i \le 2, \beta_j \le 4$ for $i \le k$, $j \le m$.
\end{theorem}
The main tool connecting graph polynomials with list colourings is Combinatorial Nullstellensatz \cite{combnull}. It implies that for every non-vanishing monomial of $P(G)$, if we assign to each vertex of $G$ a list of length greater than the exponent of corresponding variable in that monomial, then such list assignment admits a proper colouring.

We note that this approach can be continued, allowing one to obtain stronger equivalents of already known results for list-colouring. Moreover, in~\cite{GrytczukZhu} where it is proven that every planar graph $G$ contains a matching $M$ such that $AT(G-M) \le 4$, one can find an example that with this approach it is possible to get results that are not known (or hard to prove) for ordinary list colouring.

In this paper we provide a graph polynomial analogue to the result of Hutchinson, obtaining a characterisation of polynomial extendability for outerplanar graphs, which may be presented in the form of the following theorem.
\begin{theorem} \label{thm:general2}
Let $G$ be any outerplanar graph with $V(G) = 
\{x, y, v_1, \dots, v_n\}$. Then in $P(G)$ there is a non-vanishing monomial of the form $\eta x^\beta y^\gamma \prod_{i=1}^n v_i^{\alpha_i}$ with $\alpha_i \le 2$, $\beta, \gamma \le 1$ satisfying:
\begin{enumerate}
\item $\beta = \gamma = 1$ when every proper $3$-colouring $C$ of $G$ forces $C(x) = C(y)$;
\item $\beta + \gamma = 1$ when every proper $3$-colouring $C$ of $G$ forces  $C(x) \neq C(y)$; 
\item $\beta = \gamma = 0$ otherwise.
\end{enumerate}
\end{theorem}

\begin{figure}[htb]
	\begin{center}
		\includegraphics[width = 0.9\textwidth]{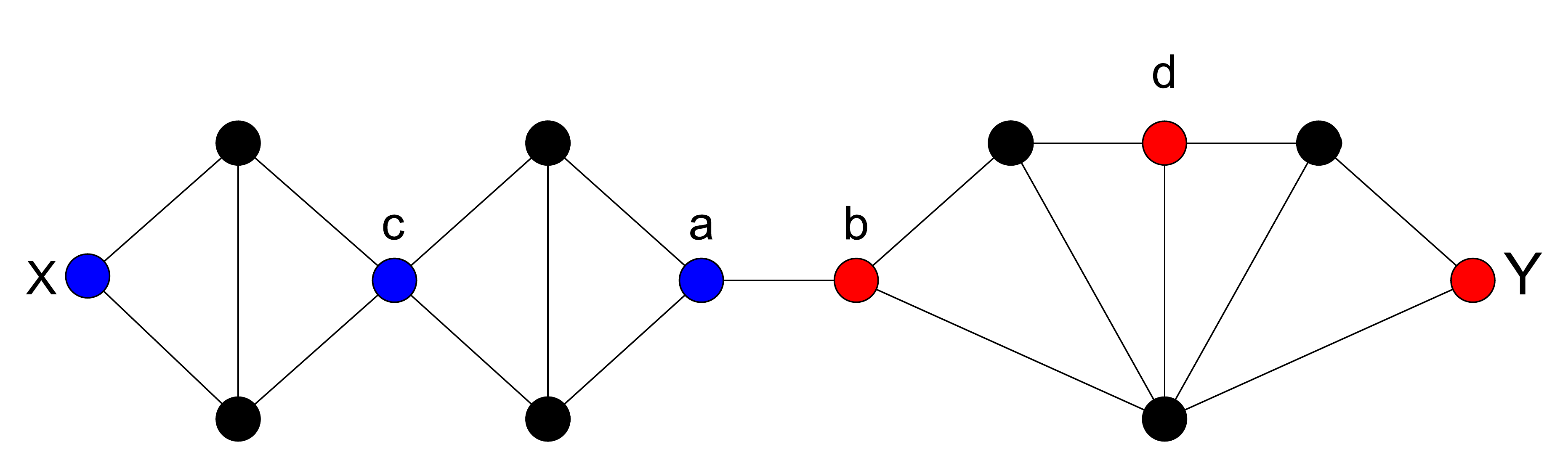}
	\end{center}
	\caption{An example of a graph satisfying conditions of point $ii)$ of Theorem~\ref{thm:general2}. When 3-colouring the graph, vertices $a$ and $b$ need to be in different colours. Vertices $x$ and $c$ are in the same colour class as $a$ (an example of the chain of diamonds), while $y$ and $d$ are in the same colour class as $b$ (the diamonds are linked along an edge). Therefore $x$ and $y$ have different colours in every proper 3-colouring of the graph. The black vertices are yet to be coloured.} \label{fig:colouring}
\end{figure}

We note that our proofs are simpler than the ones of Hutchinson, which show the strength of the graph polynomial method for graph colouring problems. All considered graphs are simple, undirected, and finite. For background in graph theory see~\cite{West}. 

\section{Outerplane near-triangulations}
In this section we provide a graph polynomial analogue to Theorem~\ref{thm:hutchinson}. The main tool is the following theorem.

\begin{theorem}\label{thm:algebraic}
		Let $G$ be a triangle or any $2$-connected, outerplane near-triangulation with exactly two vertices of degree 2. Let $z \in V(G)$ be any neighbour of a degree 2 vertex. Denote $V(G)=\lbrace x, y, z, v_1, \dots, v_n \rbrace$, where $deg(x)=deg(y)=2$, $yz \in E(G)$, $y$, $z$ $\ne x$. Then
		 $$P(G)=Q(G) + \eta_1 x v_1^2 \dots v_n^2 y^0 z^2 + \eta_2 x v_1^2 \dots v_n^2 y^1 z^1 + \eta_3 x v_1^2 \dots v_n^2 y^2 z^0,$$ where $\lbrace \eta_1, \eta_2, \eta_3 \rbrace = \lbrace -1,0,1 \rbrace$, while $Q(G)$ is a sum of monomials of the form $\eta x^{\alpha_x} v_1^{\alpha_1} \dots v_n^{\alpha_n} y^{\alpha_y} z^{\alpha_z}$, $\eta \ne 0$, with $(\alpha_x, \alpha_1, \dots, \alpha_n) \ne (1,2, \dots, 2)$.
	\end{theorem}
	
	\begin{proof}
		The proof is done by induction on $n$. For the base step $(n=0)$, let $G$ be a triangle on vertices $\lbrace x,y,z \rbrace$. It is easy to check, that:
		\begin{center}
			$P(G) = (x - y)(y - z)(x - z) = x^2 y^1 z^0 - x^2 y^0 z^1 + x^1 y^0 z^2 - x^1 y^2 z^0 + x^0 y^2 z^1 - x^0 y^1 z^2  = Q(G) + x^1 y^0 z^2 - x^1 y^2 z^0$,
		\end{center}
		hence we have $\eta_2 = 0$ and $\lbrace \eta_1, \eta_3 \rbrace = \lbrace 1, -1 \rbrace$, and with $Q(G)$ having necessary form, $G$ is concordant with the thesis.
		
		We now proceed with the induction. Let $n \in \N$ and suppose the theorem holds for graphs on at most $n+3$ vertices. Let $G'$ be any 2-connected, outerplane near-triangulation on $n+4$ vertices and $x, y \in V(G')$ be the only two vertices of degree 2. Notice that $x$ and $y$ cannot be adjacent (their common neighbour would then be a cutvertex, thus violating 2-connectivity). Let $z$ and $v_{n+1}$ be the neighbours of $y$. There is $deg(z), deg(v_{n+1}) \ge 3$ and (because $G'$ is triangulated) $zv_{n+1} \in E(G')$. Now consider $G = G' - y$. Note that $G$ remains 2-connected outerplane near-triangulation. As outerplanar graph should have at least 2 vertices of degree at most 2, one of neighbours of $y$ has now degree 2. Let us name it $\strange{y}$, while the second one --- $\strange{z}$. Notice that due to triangularity and 2-connectiveness, we have $deg(\strange{z}) > 2$ (with an exception when $G$ is a triangle), as $\strange{y}$ and $\strange{z}$ have a common neighbour.
		Now, we may consider $P(G)$ using the inductive assumption. There are three possible cases:

	\noindent \textit{1. $\strange{\eta}_1 = 0$}. As $\strange{\eta}_1 = 0$ and $\lbrace \strange{\eta}_2, \strange{\eta}_3 \rbrace = \lbrace -1, 1 \rbrace$, we know that:
		\begin{center}
			$P(G) = Q(G) +  \strange{\eta}_2 x v_1^2 \dots v_n^2 \strange{y}^1 \strange{z}^1 + \strange{\eta}_3 x v_1^2 \dots v_n^2 \strange{y}^2 \strange{z}^0 = Q(G) + \strange{\eta}_2 x v_1^2 \dots v_n^2 \strange{y}^1 \strange{z}^1 - \strange{\eta}_2 x v_1^2 \dots v_n^2 \strange{y}^2 \strange{z}^0 =$\\
			$Q(G) + \strange{\eta}_2 x v_1^2 \dots v_n^2(\strange{y}^1\strange{z}^1 - \strange{y}^2\strange{z}^0).$
		\end{center}
		Now, $P(G') = P(G) (\strange{y}-y) (\strange{z}-y) = P(G)(\strange{y}\strange{z} - \strange{y}y - \strange{z}y +y^2)$, thus:
		\begin{center}
			$P(G') = (Q(G) + \strange{\eta}_2 x v_1^2 \dots v_n^2(\strange{y}^1\strange{z}^1 - \strange{y}^2\strange{z}^0))(\strange{y}\strange{z} - \strange{y}y - \strange{z}y +y^2) = Q(G)(\strange{y}\strange{z} - \strange{y}y - \strange{z}y +y^2) + \strange{\eta}_2 x v_1^2 \dots v_n^2(\strange{y}^2 \strange{z}^2 y^0 - \strange{y}^2 \strange{z}^1 y^1 - \strange{y}^1 \strange{z}^2 y^1 + \strange{y}^1 \strange{z}^1 y^2 - \strange{y}^3 \strange{z}^1 + \strange{y}^3 y^1 + \strange{y}^2 \strange{z}^1 y^1 - \strange{y}^2 \strange{z}^0 y^2) = Q'(G') + \strange{\eta}_2 x v_1^2 \dots v_n^2(\strange{y}^2 \strange{z}^2 y^0 - \strange{y}^1 \strange{z}^2 y^1 - \strange{y}^2 \strange{z}^0 y^2)$.
		\end{center}
		Now either $z = \strange{y}$ and $v_{n+1} = \strange{z}$, respectively, or the inverse may occur. In the first case, we have:
		\begin{center}
			$P(G') = Q'(G') + \strange{\eta}_2 x v_1^2 \dots v_n^2(v_{n+1}^2 y^0 z^2 - v_{n+1}^2 y^1 z^1 - v_{n+1}^0 y^2 z^2)$,
		\end{center}	 
		thus $\lbrace \eta_1, \eta_2 \rbrace = \lbrace -1, 1 \rbrace$ and $\eta_3 = 0$, with the last monomial going into $Q'(G')$. With analogous calculations, in the second case we have $\lbrace \eta_1, \eta_3 \rbrace = \lbrace -1, 1 \rbrace$ and $\eta_2 = 0$. As $Q'(G')$ obviously contains only monomials of the form $\eta x^{\alpha_x} v_1^{\alpha_1} \dots v_{n+1}^{\alpha_{n+1}} y^{\alpha_y} z^{\alpha_z}$, $\eta \ne 0$, $(\alpha_x, \alpha_1, \dots, \alpha_{n+1}) \ne (1,2, \dots, 2)$, it can assume the role of $Q(G)$, and the case is finished.
		 
	\noindent \textit{2. $\strange{\eta}_2 = 0$}. As $\strange{\eta}_2 = 0$ and $\lbrace \strange{\eta}_1, \strange{\eta}_3 \rbrace = \lbrace -1, 1 \rbrace$, we know that:
		\begin{center}
			$P(G) = Q(G) +  \strange{\eta}_1 x v_1^2 \dots v_n^2 \strange{y}^0 \strange{z}^2 + \strange{\eta}_3 x v_1^2 \dots v_n^2 \strange{y}^2 \strange{z}^0 = Q(G) + \strange{\eta}_1 x v_1^2 \dots v_n^2 \strange{y}^0 \strange{z}^2 - \strange{\eta}_1 x v_1^2 \dots v_n^2 \strange{y}^2 \strange{z}^0 =$\\
			$Q(G) + \strange{\eta}_1 x v_1^2 \dots v_n^2(\strange{y}^0\strange{z}^2 - \strange{y}^2\strange{z}^0).$
		\end{center}
		And then:
		\begin{center}
			$P(G') = (Q(G) + \strange{\eta}_1 x v_1^2 \dots v_n^2(\strange{y}^0\strange{z}^2 - \strange{y}^2\strange{z}^0))(\strange{y}\strange{z} - \strange{y}y - \strange{z}y +y^2) = Q(G)(\strange{y}\strange{z} - \strange{y}y - \strange{z}y +y^2) + \strange{\eta}_1 x v_1^2 \dots v_n^2(\strange{y}^1 \strange{z}^3 y^0 - \strange{y}^1 \strange{z}^2 y^1 - \strange{y}^0 \strange{z}^3 y^1 + \strange{y}^0 \strange{z}^2 y^2 - \strange{y}^3 \strange{z}^1 y^0 + \strange{y}^3 \strange{z}^0 y^1 + \strange{y}^2 \strange{z}^1 y^1 - \strange{y}^2 \strange{z}^0 y^2) = Q'(G') + \strange{\eta}_1 x v_1^2 \dots v_n^2(\strange{y}^0 \strange{z}^2 y^2 - \strange{y}^1 \strange{z}^2 y^1 - \strange{y}^2 \strange{z}^0 y^2 + \strange{y}^2 \strange{z}^1 y^1)$.
		\end{center}
	Continuing as in case 1, when $z =\strange{y}$ and $v_{n+1}= \strange{z}$, respectively, we have $\lbrace \eta_2, \eta_3 \rbrace = \lbrace -1, 1 \rbrace$ and $\eta_1 = 0$. In the inverse case, when $v_{n+1} = \strange{y}$ and $z = \strange{z}$, there is $\lbrace \eta_2, \eta_3 \rbrace = \lbrace 1, -1 \rbrace$ and $\eta_1 = 0$. $Q'(G')$ can again assume the role of $Q(G)$, and this case is also done.
	
	\noindent \textit{3. $\strange{\eta}_3 = 0$}. This case is handled analogously as $\strange{\eta}_1 = 0$, interchanging the roles of $\strange{y}$ and $\strange{z}$. 
		Here we have:
		\begin{center}
			$P(G) = Q(G) + \strange{\eta}_1 x v_1^2 \dots v_n^2 \strange{y}^0 \strange{z}^2 +  \strange{\eta}_2 x v_1^2 \dots v_n^2 \strange{y}^1 \strange{z}^1 = Q(G) + \strange{\eta}_2 x v_1^2 \dots v_n^2 \strange{y}^1 \strange{z}^1 - \strange{\eta}_2 x v_1^2 \dots v_n^2 \strange{y}^0 \strange{z}^2 =$\\
			$Q(G) + \strange{\eta}_2 x v_1^2 \dots v_n^2(\strange{y}^1\strange{z}^1 - \strange{y}^0\strange{z}^2).$
		\end{center}
	And then:
	\begin{center}
		$P(G') = (Q(G) + \strange{\eta}_2 x v_1^2 \dots v_n^2(\strange{y}^1\strange{z}^1 - \strange{y}^0\strange{z}^2))(\strange{y}\strange{z} - \strange{y}y - \strange{z}y +y^2) = Q(G)(\strange{y}\strange{z} - \strange{y}y - \strange{z}y +y^2) + \strange{\eta}_2 x v_1^2 \dots v_n^2(\strange{y}^2 \strange{z}^2 y^0 - \strange{y}^2 \strange{z}^1 y^1 - \strange{y}^1 \strange{z}^2 y^1 + \strange{y}^1 \strange{z}^1 y^2 - \strange{y}^1 \strange{z}^3 + \strange{y}^1 \strange{z}^2 y^1 + \strange{z}^3 y^1 - \strange{y}^0 \strange{z}^2 y^2) = Q'(G') + \strange{\eta}_2 x v_1^2 \dots v_n^2(\strange{y}^2 \strange{z}^2 y^0 - \strange{y}^2 \strange{z}^1 y^1 - \strange{y}^0 \strange{z}^2 y^2)$.
	\end{center}
		 Finally, when $z =\strange{y}$ and $v_{n+1}= \strange{z}$, respectively, we have $\lbrace \eta_1, \eta_3 \rbrace = \lbrace -1, 1 \rbrace$ and $\eta_2 = 0$. In the inverse case, when $v_{n+1} = \strange{y}$ and $z = \strange{z}$, there is $\lbrace \eta_1, \eta_2 \rbrace = \lbrace -1, 1 \rbrace$ and $\eta_3 = 0$.

Therefore, in each case we have the desired form of the polynomial, thus completing the inductive argument.
	\end{proof}

Recall that by Combinatorial Nullstellensatz, $(i, j)$-extendability of $(G,x,y)$ can be expressed as the fact that there is a non-vanishing monomial in $P(G)$ where exponents of $x$ and $y$ are $i-1$ and $j-1$, respectively, and every other exponent is less than 3.
We obtain an analogue to Theorem~\ref{thm:hutchinson} as the following
\begin{corollary}\label{cor:Hutch}
Let $G$ be any $2$-connected, outerplane near-triangulation with $V(G)=\lbrace  x, y, v_1, \dots, v_n \rbrace$. Let $C \colon V(G) \to \{1,2,3\}$ be any proper 3-colouring of $G$. Then in the graph polynomial $P(G)$
\begin{enumerate}
\item there is no monomial of the form $\eta x^0 y^0 \prod_{i=1}^n v_i^{\alpha_i}$ with $\alpha_i \le 2$;
\item the monomial of the form $\eta x^1 y^0 \prod_{i=1}^n v_i^{\alpha_i}$ with $\alpha_i \le 2$ does not vanish if and only if $C(x) \neq C(y)$;
\item there is non-vanishing monomial of the form $\eta x^\beta y^\gamma \prod_{i=1}^n v_i^{\alpha_i}$ with $\alpha_i \le 2$, $\beta, \gamma \le 1$.
\end{enumerate}
\end{corollary}
\begin{proof}
For the first point, simply note that outerplane near-triangulation on $n+2$ vertices has $2n+1$ edges, while the sum of the exponents of the given monomial is at most $2n$. 

For the second point and for the third one: when $x$ and $y$ are adjacent one may apply Theorem~\ref{thm:zhu} directly; otherwise, by the Hutchinson's shrinking argument it is enough to verify an existence of a suitable monomial for $G$ having exactly 2 vertices of degree 2, when $x$ and $y$ are these vertices.  

Indeed, suppose otherwise and consider any chord (inner edge) $ab$ of $G$ that separates the component $H$ of the graph not containing vertices $x$ and $y$. Such a chord exists, unless $x$ and $y$ are the only degree 2 vertices of $G$. Let $G_1 = G[V(G) \setminus V(H)]$ and $G_2 = G[V(H) \cup \{a, b\}]$. By Theorem~\ref{thm:zhu} $P(G_2 - ab)$ contains non-vanishing monomial of the form $s_2 = \eta a^0 b^0 v_1^{\alpha_1} \dots v_k^{\alpha^k}$ with $\alpha_i \le 2$. 
Note, that common variables in $P(G_1)$ and $P(G_2 - ab)$ are $a$ and $b$ only and that the sum of the exponents in any monomial in $P(G_2 - ab)$ is fixed. Hence, any other monomial in $P(G_2 - ab)$ has different exponents for some of $v_1, \dots v_k$. Therefore, as there is $P(G) = P(G_1) P(G_2 - ab)$, $G$ with $x$ and $y$ satisfies the second (or the third one, respectively) point of the corollary if and only if $G_1$ with $x$ and $y$ does. Actually, the existence of desired monomials $s$ in $P(G)$ and $s_1$ in $P(G_1)$, respectively, is equivalent by identity $s = s_1 s_2$.

Repeating the above argument until there is no separating chord one can shrink $G$ to the claimed form. By Theorem~\ref{thm:algebraic} this finishes the proof of the third point as then one has either $\eta_1 \neq 0$ or $\eta_2 \neq 0$. For the second point it is enough to notice that under the assumption of Theorem~\ref{thm:algebraic} there is $\eta_1 = 0$ if and only if $C(x) = C(y)$. Note that there is also $\eta_3 = 0$ if and only if $C(z) = C(x)$ and then $\eta_2 = 0$ if and only if $x, y$ and $z$ have 3 different colours. One may prove this fact by a simple analysis of the inductive step in the proof of Theorem~\ref{thm:algebraic}. 

Indeed, in the base case (a triangle $xyz$) we have $\eta_2 = 0$. Further, when $G$ is extended to $G'$ by a triangle $\strange{y} \strange{z}y$ then
\begin{enumerate}
\item[1.] $\strange{\eta}_1 = 0$ ($C(\strange{y}) = C(x)$) forces $\eta_3 = 0$ (when $z = \strange{y}$) or $\eta_2 = 0$ (when $z = \strange{z}$),
\item[2.] $\strange{\eta}_2 = 0$ forces $\eta_1 = 0$ ($C(x) = C(y)$),
\item[3.] $\strange{\eta}_3 = 0$ ($C(\strange{z}) = C(x)$) forces $\eta_3 = 0$ (when $z = \strange{z}$) or $\eta_2 = 0$ (when $z = \strange{y}$)     .
\end{enumerate}
\end{proof}

\section{Poly-extendability of general outerplanar graphs}
The results of the previous section can be of course applied to any outerplanar graph, not necessarily triangulated. This, however, leads to loss of information, as usually there is more than one way to triangulate the graph, and different triangulations may lead to different types of extendability. Moreover, in the case of non-triangulated graphs, as well as those that are not 2-connected, the counting argument behind point (i) of Corollary~\ref{cor:Hutch} does not work any more. Hence, it is possible for a general outerplanar graph to be $(1,1)$-extendable.
At first, a formal definition of fundamental subgraphs is provided, followed by three instrumental lemmas.

\begin{definition}
	Let $G$ be a 2-connected outerplane graph, $x,y \in V(G)$ and let $T(G)$ be the weak dual of $G$. The \emph{fundamental $x-y$ subgraph of $G$} is the subgraph of $G$ induced by the vertices belonging to faces that have vertices representing them in $T(G)$ lying on the shortest path between vertices representing faces on which $x$ and $y$ lie. If $xy \in E(G)$, then the fundamental subgraph reduces to an edge $xy$.
\end{definition}

Here, the assumption that the graph is outerplane is needed, as the construction of weak dual requires a particular embedding to be chosen. Notice however that in case of 2-connected outerplanar graphs there is, up to isomorphism, just one outerplane embedding, hence every 2-connected outerplanar graphs has essentially a single weak dual. Therefore in the rest of the paper we will assume the graphs to be outerplanar, as the choice of an embedding is irrelevant for our purpose.

\begin{definition}
	Let $G$ be a connected outerplanar graph with cutvertices, and let $BC(G)$ be the block-cutvertex graph of $G$. Let $x, y \in V(G)$ be vertices lying in two different blocks of $G$. The fundamental $x-y$ subgraph of $G$ consists of all blocks that have vertices representing them in $BC(G)$ lying on the shortest path between vertices representing blocks containing $x$ and $y$, and each of those blocks is restricted to the fundamental $a-b$ subgraph, where $a, b \in V(G)$ are the two cutvertices belonging to the given block and to the shortest path between blocks containing $x$ and $y$ in $BC(G)$.
\end{definition}

\begin{definition}
	An outerplanar graph $G$ with $x, y \in V(G)$ is \emph{$xy$-fundamental} if its fundamental $x-y$ subgraph is equal to $G$. An outerplanar graph $G$ is \emph{fundamental} if it is $xy$-fundamental for some $x,y \in V(G)$.
\end{definition}
\begin{figure}[htb]
\begin{center}
	\includegraphics[width = 0.7\textwidth]{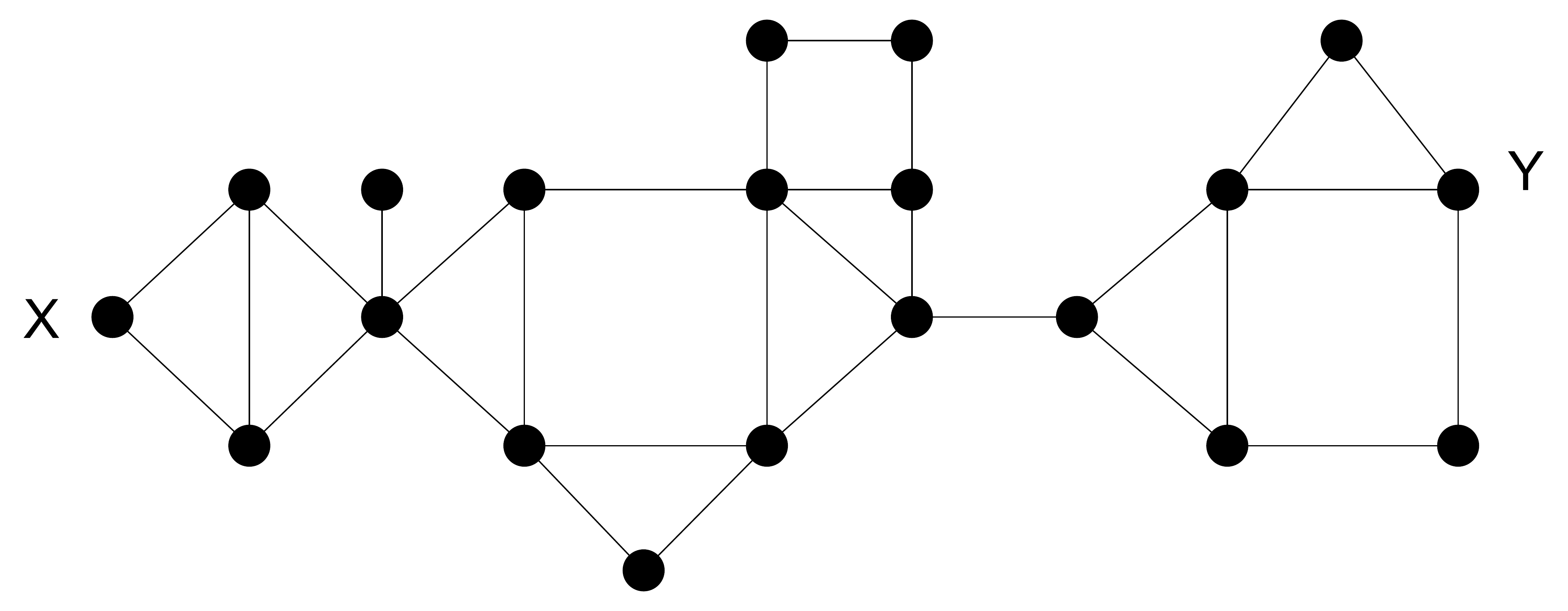}

	\includegraphics[width = 0.7\textwidth]{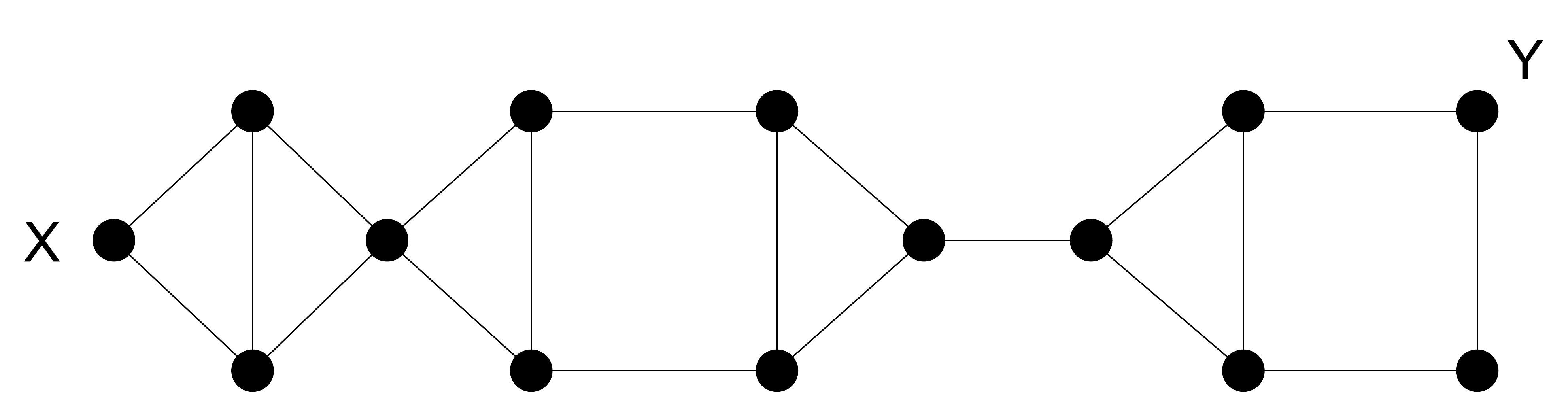}
\end{center}
\caption{Top: a connected, outerplanar graph $G$; Bottom: a fundamental $x-y$ subgraph of $G$.} \label{fig:fundamental}
\end{figure}
\begin{lemma}\label{lem:l1}
	Let $G$ be a 2-connected $xy$-fundamental near-triangulation, such that $C(x)=C(y)$, where $C \colon V(G) \to \{1,2,3\}$ is any proper 3-colouring of $G$. Let $v_0$ be the vertex of $G$ that has degree 2 in $G-y$, and $v_1, \dots, v_n$ be the remaining vertices. Then in $P(G)$ there is a non-vanishing monomial of the form $\eta x^0y^2v_0^1v_1^2 \dots v_n^2$, with $\eta \in \lbrace -1,1 \rbrace$.
\end{lemma}	
\begin{proof}
	As $C(x)=C(y)$, then $C(x) \neq C(v_0)$. Hence by the second case of Corollary~\ref{cor:Hutch}, there is a non-vanishing monomial $\eta x^0v_0^1v_1^2 \dots v_n^2$, with $\eta \in \lbrace -1,1 \rbrace$ in $P(G-y)$. Adding $y$ back, thus multiplying $P(G-y)$ by $(y-v_0)(y-v_n)=y^2-yv_0-yv_n+v_0v_n$, we get the monomial specified in thesis, and as it is the only way to obtain it, it is non-vanishing.
\end{proof}

\begin{lemma}\label{lem:l2}
	Let $G,G'$ be any two graphs, such that $V(G) = \lbrace x, v_1, \dots, v_n \rbrace $, $V(G') = \lbrace x', u_1, \dots, u_m \rbrace$. Let $G''$ be the graph obtained from $G$ and $G'$ by identifying $x$ with $x'$, thus creating vertex $x''$, and carrying neighbouring relations from $G, G'$. Suppose there are non-vanishing monomials $\eta x^\alpha \Pi v_i^{\alpha_i}$ and $\eta' x'^\beta \Pi u_j^{\beta_j}$ in $P(G)$ and $P(G')$ respectively. Then in $P(G'')$ there is a non-vanishing monomial $A(G'') = \eta \eta' x''^{\alpha + \beta}  \Pi v_i^{\alpha_i} \Pi u_j^{\beta_j}$.
\end{lemma}
\begin{proof}
	As both $\eta$ and $\eta'$ are non-zero, then the only way $A(G'')$ would vanish is that there were a monomial $A'(G'') = \nu \nu' x''^{\alpha' + \beta'}  \Pi v_i^{\alpha_i} \Pi u_j^{\beta_j}$, where $\nu \nu' = -\eta \eta'$ and $\alpha' + \beta' = \alpha + \beta$. But then in $P(G)$ and $P(G')$ there would have to be respective non-vanishing monomials $\nu x^{\alpha'} \Pi v_i^{\alpha_i}$ and $\nu' x'^{\beta'} \Pi u_j^{\beta_j}$, and as the sum of exponents in every monomial in a polynomial of given graph is fixed, we have that $\alpha = \alpha'$ and $\beta = \beta'$, a contradiction. Thus $A(G'')$ is non-vanishing.
\end{proof}
\begin{figure}[htb]
\begin{center}
	\begin{minipage}{0.45\linewidth}
	\includegraphics[width = 0.9\textwidth]{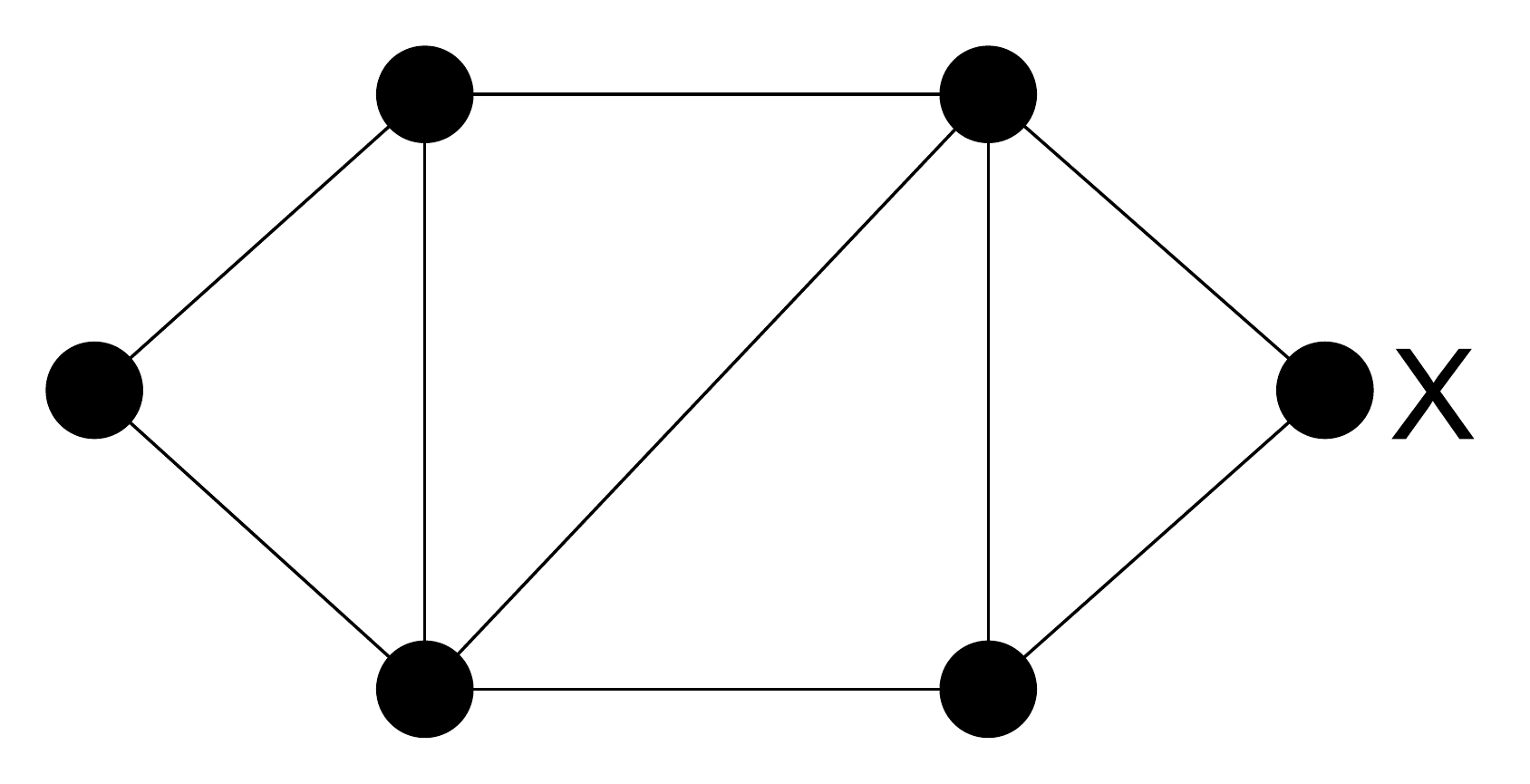}
	\end{minipage}
	\begin{minipage}{0.45\linewidth}
	\includegraphics[width = 0.9\textwidth]{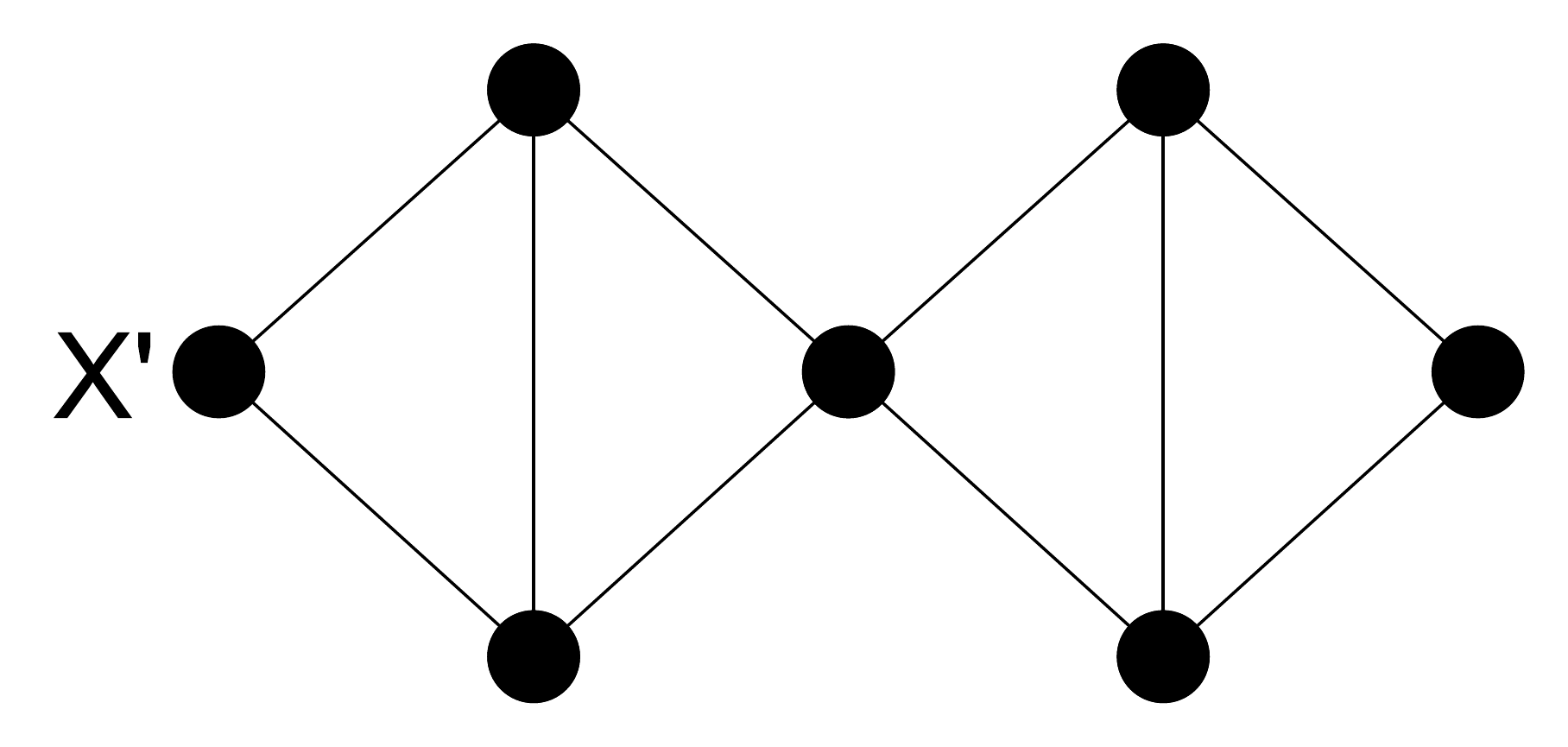}
	\end{minipage}
\end{center}
\begin{center}
	\includegraphics[width = 0.7\textwidth]{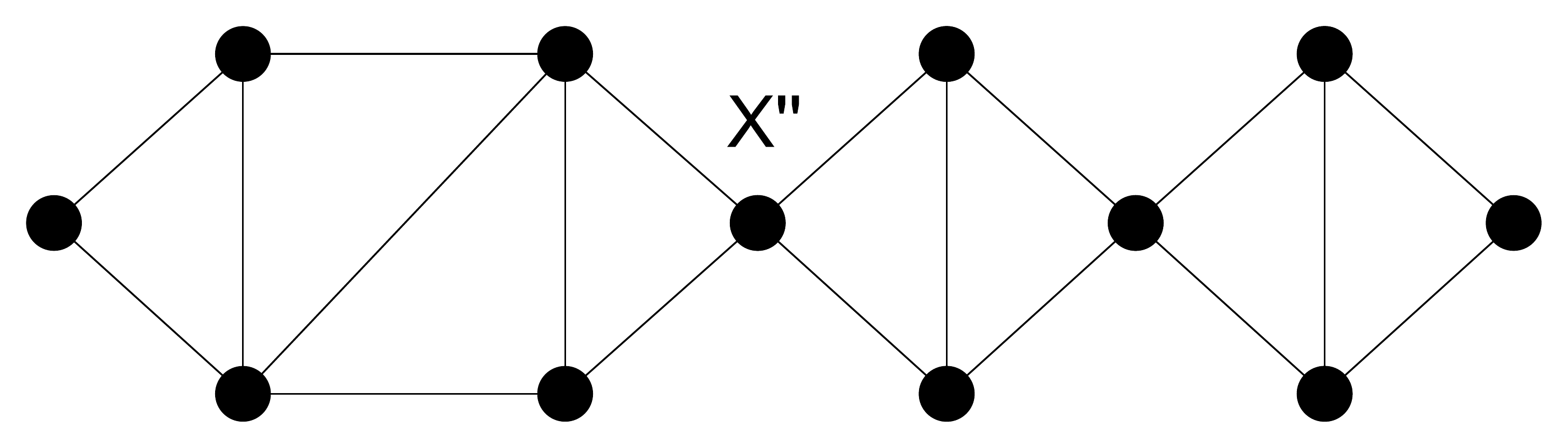}
\end{center}
\caption{Illustration for lemma~\ref{lem:l2}. Top: graphs $G$ (left) and $G'$ (right); Bottom: graph $G''$.}
\label{fig:lem12}
\end{figure}
\begin{lemma}\label{le:path}
	Let $G$ be a path of length $n$, $n \ge 2$, where $x, y$ are the endpoints and $v_1, \dots, v_{n-1}$ are the internal vertices of $G$. Then in $P(G)$ there is a non-vanishing monomial of the form $\eta x^0y^0v_1^2 v_2^1 \dots v_{n-1}^1$, where $\eta \in \lbrace -1,1 \rbrace$.
\end{lemma}
\begin{proof}
	Suppose at first that $n=2$. Then $P(G)=(x-v_1)(y-v_1)=xy-xv_1-yv_1+v_1^2$, and the last monomial is the one fulfilling the thesis. Now suppose that the lemma holds for $n=k-1$. Then in $P(G)$, where $G$ is a path $x v_1 \dots v_{k-1}$, there is a monomial $\eta x^0v_1^2 v_2^1 \dots v_{k-2}^1v_{k-1}^0$. Now adjoining $v_k$ to $v_{k-1}$, thus multiplying $P(G)$ by $(v_{k-1}-v_k)$ we obtain a monomial $\eta x^0v_1^2 v_2^1 \dots v_{k-2}^1v_{k-1}^1v_k^0$ for path of length $k$, hence completing the induction.
\end{proof}

\vspace{1em}

\subsection{Near-triangulations with cutvertices}~\\

The following theorem is a polynomial analogue of \cite[Theorem 5.3]{Hutchinson} that characterizes extendability of outerplanar near-triangulations with cutvertices. 

\begin{theorem}\label{thm:cutvertices}
	Let $G$ be a fundamental $x-y$ subgraph with cutvertices $\lbrace v_1, \dots, v_{j-1} \rbrace$, $CV(G)= (x, v_1, \dots, v_{j-1}, y)$ be the sequence consisting of $x$, $y$ and the cutvertices of $G$ in order that they occur on any of the paths from $x$ to $y$, and $u_{i,k}$ being the remaining vertices in the $i$-th block. Then in $P(G)$:
	\begin{enumerate}
		\item there is a non-vanishing monomial of the form $\eta_1 x^1 y^1 \Pi v_m^{\alpha_m} \Pi u_{i,k}^{\beta_{i,k}}, \alpha_m, \beta_{i,k} \le 2$ if every vertex from $CV(G)$ is in the same colour class; 
		\item there is a non-vanishing monomial of the form $\eta_2 x^0 y^1 \Pi v_m^{\alpha_m} \Pi u_{i,k}^{\beta_{i,k}}, \alpha_m, \beta_{i,k} \le 2$ if there is a single pair of successive vertices in $CV(G)$ that are in different colour classes;
		\item there is a non-vanishing monomial of the form $\eta_3 x^0 y^0 \Pi v_m^{\alpha_m} \Pi u_{i,k}^{\beta_{i,k}}, \alpha_m, \beta_{i,k} \le 2$ if there are at least two pairs of successive vertices in $CV(G)$ that are in different colour classes;
	\end{enumerate}
\end{theorem}
\begin{figure}[htb]
\begin{center}
	\includegraphics[width = 0.9\textwidth]{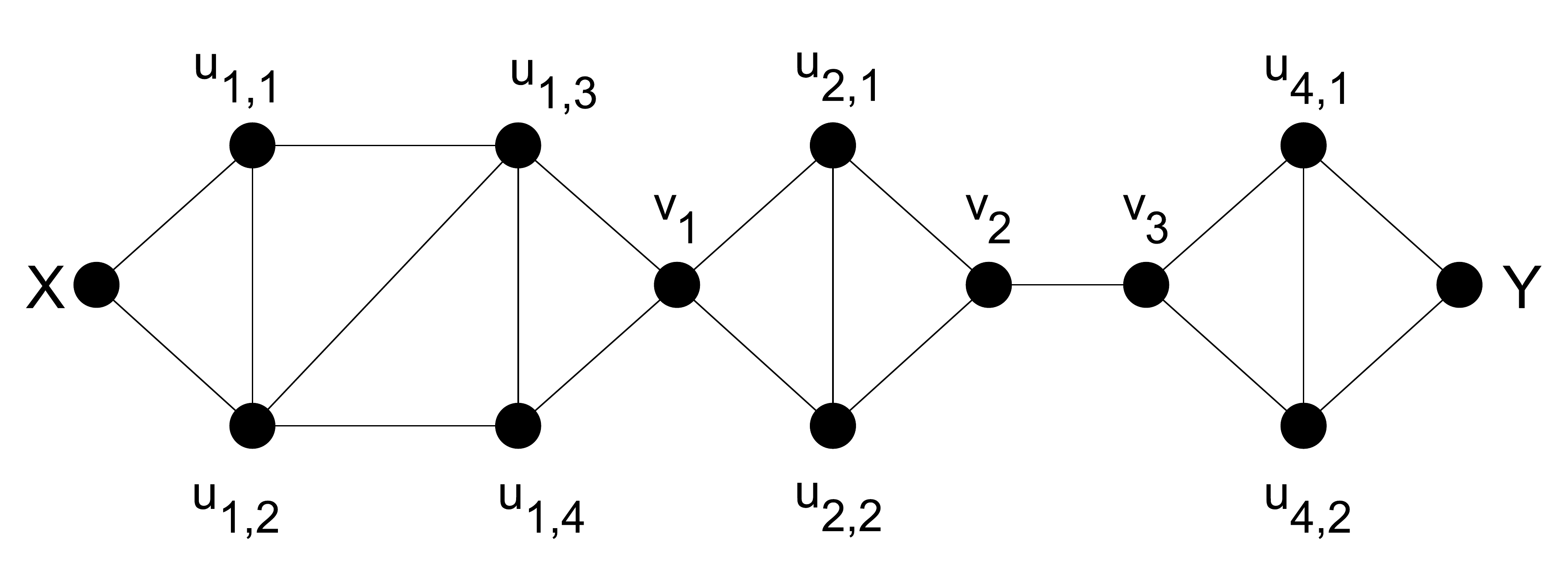}
\end{center}
\caption{An example of labelling described in Theorem~\ref{thm:cutvertices}.}
\label{fig:cutvertices}
\end{figure}
\begin{proof}
	Start with partitioning $G$ by its cutvertices into separate, 2-connected, $v_{i-1}v_i$-fundamental outerplanar near-triangulations $B_1, \dots, B_j$. To each of these graphs, Theorem~\ref{thm:algebraic} applies, and $P(G)=P(B_1) \dots P(B_j)$. If in each of those blocks the colour class of degree 2 vertices is the same, then in each of their polynomials there is a non-vanishing monomial such that exponents of degree 2 vertices are equal to 1, with other exponents no larger than 2. Thus case 1 is just a repeated use of lemma \ref{lem:l2}.
	
	In the second case, let $B_i$ be the block with degree 2 vertices in different colour classes. If $i=1$, then in $P(B_1)$ there is a non-vanishing monomial of the form $\eta_0x^0v_1^1\Pi u^2_{1,k}$. Hence again by lemma \ref{lem:l2} we get the desired monomial. If $i>1$, then we apply lemma \ref{lem:l1} to each block $B_1$ to $B_{i-1}$, thus by lemma \ref{lem:l2} obtaining monomial with $x^0$ and $v_{i-1}^2$. As $v_{i-1}$ and $v_i$ are in different colour classes, $P(B_i)$ contains a non-vanishing monomial $\eta_i v_{i-1}^0 v_i^1 \Pi u^2_{i,k}$, hence through lemma \ref{lem:l2} we finish the case.
	
	The last case is starts analogously to the second one, with $B_i, B_l, i<l$ being two blocks with endpoints in different colour classes. Let $G'$ be the $v_{i-1}v_l$-fundamental subgraph of $G$. By Theorem~\ref{thm:algebraic} there is a non-vanishing monomial in $P(B_i)$ with $v_{i-1}^0$ and $v_i^1$ and a monomial in $P(B_l)$ with $v_{l-1}^1$ and $v_l^0$. As every block between $B_i$ and $B_l$ has a monomial with endpoints in power 1, by lemmas \ref{lem:l1} and \ref{lem:l2} there is a monomial in $P(G')$ with both $v_{i-1}$ and $v_l$ in power 0. Again by lemmas \ref{lem:l1} and \ref{lem:l2} we can now adjoin remaining parts of $G$ to $G'$, with their suitable monomials creating a desired monomial in $P(G)$.
\end{proof}
~\\

\subsection{2-connected outerplanar graphs with non-triangular faces}~\\

The following three theorems are jointly analogous to \cite[Theorem 4.3]{Hutchinson}.

\begin{theorem}\label{thm:nontri0}
	Let $G$ be a 2-connected $xy$-fundamental graph with exactly one non-triangular interior face, and that face contains $x$ and does not contain $y$. Let $V(G) = \lbrace x, y, a, b, v_1, \dots, v_n \rbrace$, where $a, b$ are the two vertices of non-triangular face belonging to the neighbouring interior face. Let $C(v)$ be the colour class of vertex $v$ in the 3-colouring of the graph induced by all of the triangular faces. Then in $P(G)$:
	\begin{enumerate}
	\item there is a non-vanishing monomial of the form $\eta_1 x^0 y^1 a^{\alpha_a} b^{\alpha_b} \Pi v_i^{\alpha_i}, \alpha_k \le 2$ if $d(x,a)=1$ and $C(a)=C(y)$ OR $d(x,b)=1$ and $C(b)=C(y)$;
	\item there is a non-vanishing monomial of the form $\eta_2 x^0 y^0 a^{\alpha_a} b^{\alpha_b} \Pi v_i^{\alpha_i}, \alpha_k \le 2$ otherwise.
	\end{enumerate}
\end{theorem}

\begin{figure}[htb]
\begin{center}
	\begin{minipage}{0.49\linewidth}
		\includegraphics[width = 0.99\textwidth]{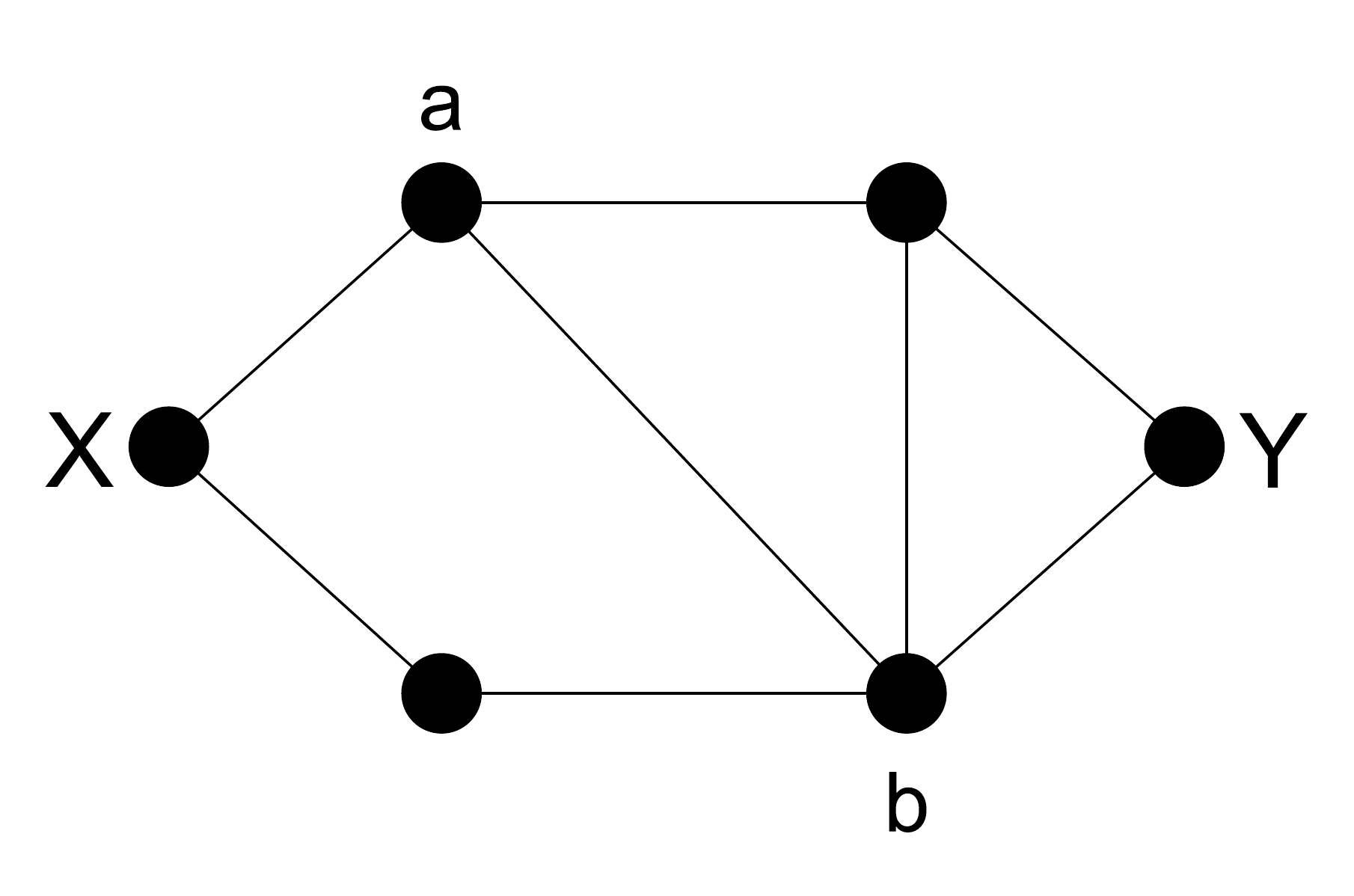}
	\end{minipage}
	\begin{minipage}{0.49\linewidth}
		\includegraphics[width = 0.99\textwidth]{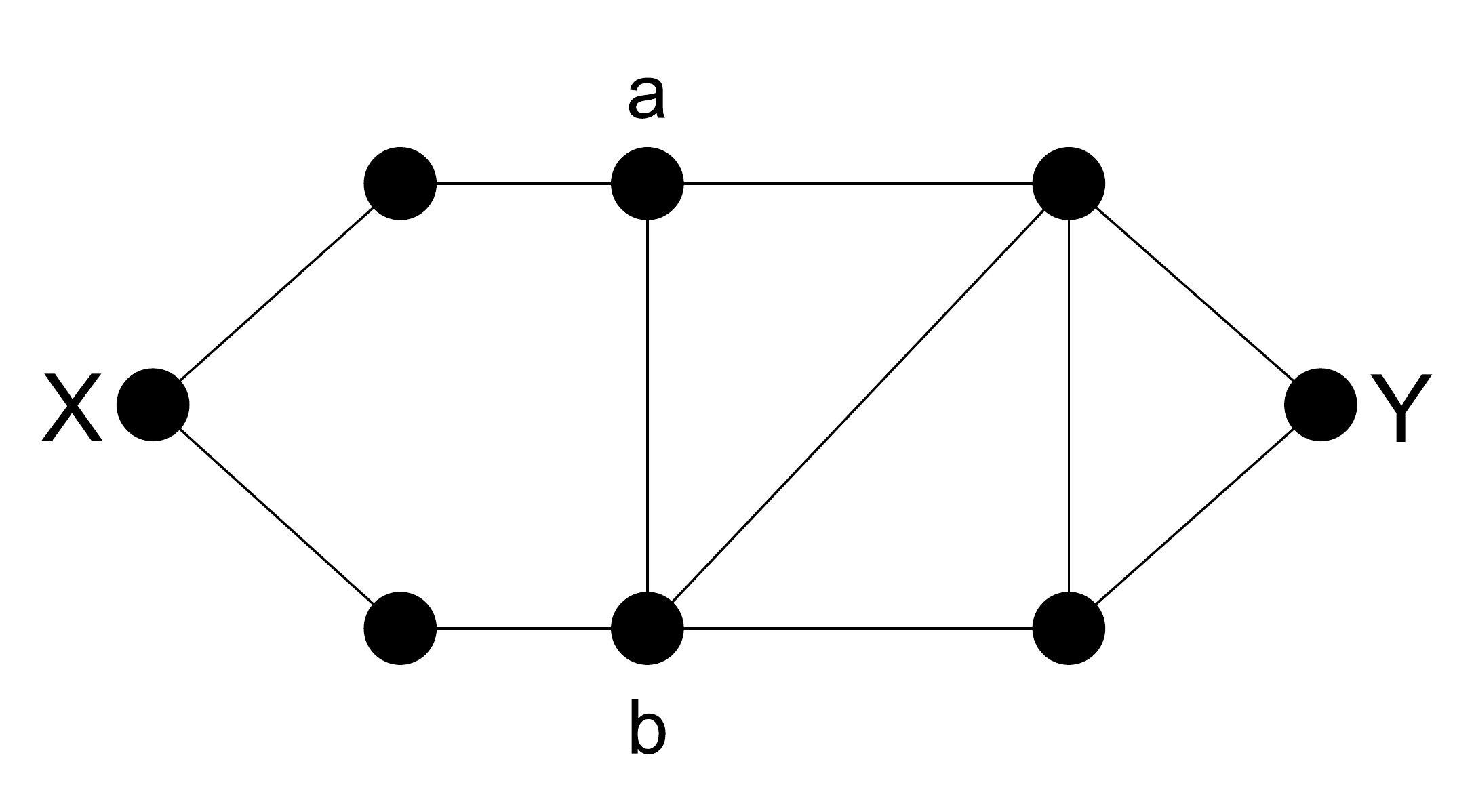}
	\end{minipage}
\end{center}
\caption{Examples of labelling as in Theorem~\ref{thm:nontri0}. Left: example to point (i); Right: example to point (ii).} \label{fig:nontri0}
\end{figure}
\begin{proof}
	Suppose that $d(x,a)=1$ and $C(a)=C(y)$. Let $G'$ be the subgraph of $G$ created by deleting all the vertices on the non-triangular face except for $a$ and $b$. As $G'$ is an outerplanar near-triangulation Theorem~\ref{thm:algebraic} applies, and as $C(a)=C(y)$, then in $P(G')$ there is a non-vanishing monomial with $a^1$ and $y^1$. If we now adjoin vertex $x$ to $a$, creating graph $G''$, then it $P(G'')$ there is a non-vanishing monomial with $x^0$, $a^2$ and $y^1$. Now adding a path between $x$ and $b$, thus reconstructing $G$ (notice that the length of this path is at least 2, as the face is not a triangle), by lemma \ref{le:path} we obtain a desired monomial. The case when $d(x,b)$ and $C(b)=C(y)$ is handled analogously.
	
	If this is not the case, then either $d(x,a) > 1$ and $d(x,b) > 1$, or $d(x,a)=1$ and $C(a) \ne C(y)$ (or analogously $d(x,b)=1$ and $C(b) \ne C(y)$). In the first case, then by Theorem~\ref{thm:algebraic} and lemma \ref{lem:l1} in $P(G')$ (with $G'$ defined as previously) there is a non-vanishing monomial with $y^0$ and all other powers less than 3. Now as we join $x$ with $a$ and $b$ with previously deleted paths, lemma \ref{le:path} gives us a monomial with $x^0, y^0$ and all other powers less than 3. In the second case, as $C(a) \ne C(y)$, by \ref{thm:algebraic} there is a monomial in $P(G')$ where $y$ has power 0 and $a$ has power 1. Adjoining $x$ to $a$, we obtain a monomial with $x^0, y^0$ and $a^2$, and as we join $x$ with $b$ by a path, lemma \ref{le:path} gives us a desired monomial. Case when $d(x,b)=1$ and $C(b) \ne C(y)$ is again analogous to the last one.
\end{proof}

\begin{theorem}\label{thm:nontri1}
	Let $G$ be a 2-connected $xy$-fundamental graph with exactly one non-triangular interior face, and that face does not contain $x$ nor $y$. Let $V(G) = \lbrace x, y, a, b, c, v_1, \dots, v_n \rbrace$, where $a, b$ and $a ,c$ are the two pairs of vertices of non-triangular face belonging to the neighbouring interior faces, and let $C(v)$ be the colour class of vertex $v$ in the 3-colouring of the subgraph of $G$ created by deleting the path connecting $b$ and $c$. Then in $P(G)$:
		\begin{enumerate}
		\item there is a non-vanishing monomial of the form $\eta_1 x^1 y^1 a^{\alpha_a} b^{\alpha_b} c^{\alpha_c} \Pi v_i^{\alpha_i}$, $\alpha_k \le 2$, if $C(x)=C(a)=C(y)$; 
		\item there is a non-vanishing monomial of the form $\eta_2 x^0 y^1 a^{\alpha_a} b^{\alpha_b} c^{\alpha_c} \Pi v_i^{\alpha_i}$, $\alpha_k \le 2$, if $C(x) \ne C(a)=C(y)$ or $C(x)=C(a) \ne C(y)$;
		\item there is a non-vanishing monomial of the form $\eta_3 x^0 y^0 a^{\alpha_a} b^{\alpha_b} c^{\alpha_c} \Pi v_i^{\alpha_i}$, $\alpha_k \le 2$, if $C(x) \ne C(a) \ne C(y)$;
	\end{enumerate}
\end{theorem}
\begin{figure}[htb]
\begin{center}
	\includegraphics[width = 0.7\textwidth]{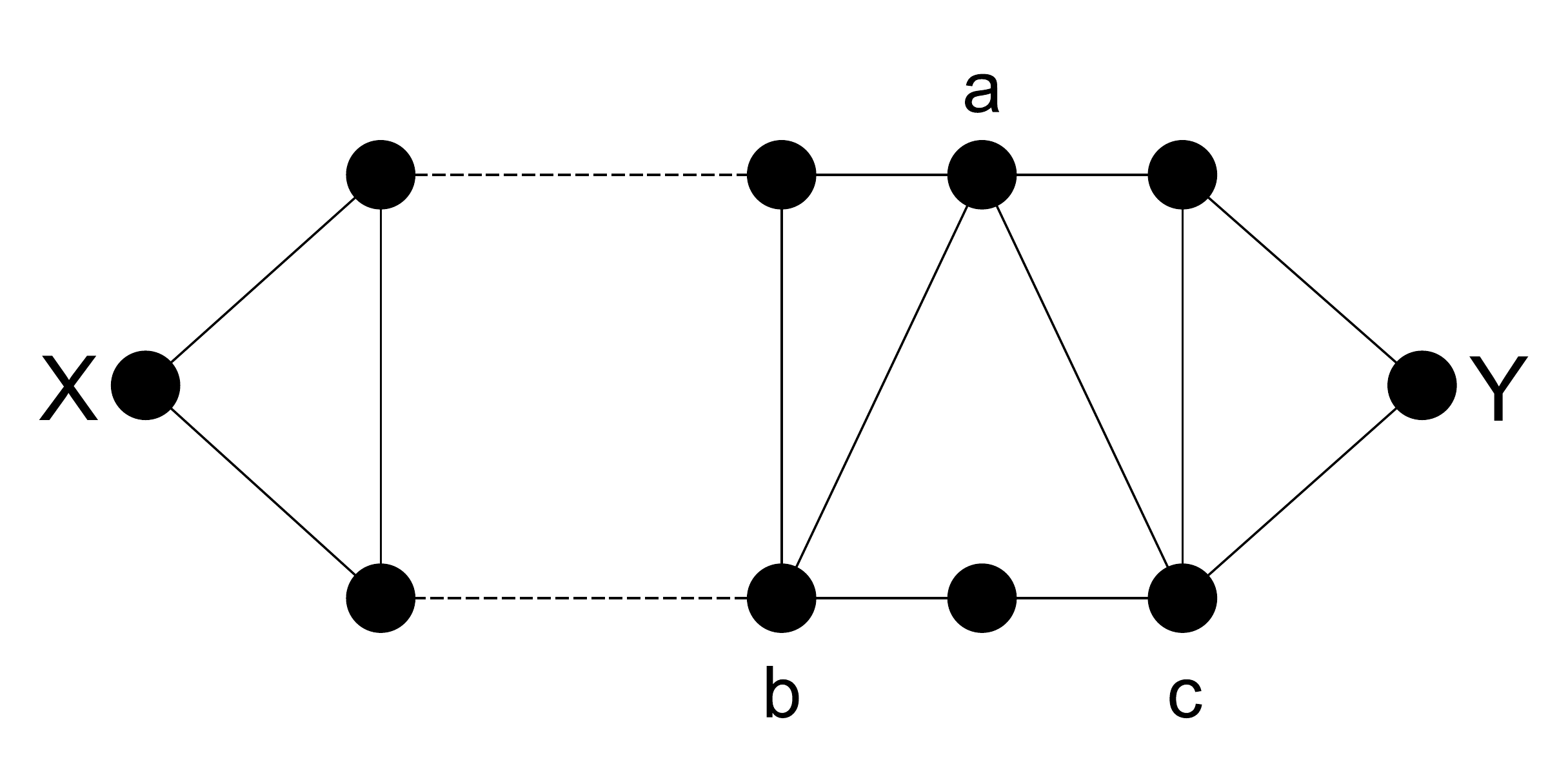}
\end{center}
\caption{An example of labelling described in Theorem~\ref{thm:nontri1}.} \label{fig:nontri1}
\end{figure}
\begin{proof}
	Let $G'$ be the subgraph of $G$ obtained by deleting path connecting $b$ and $c$ from $G$. Obviously $G'$ is an outerplanar near-triangulation with a single cutvertex $a$, hence Theorem~\ref{thm:cutvertices} applies to it. Notice moreover, that the first case of the above theorem leads to the first case of Theorem~\ref{thm:cutvertices}, and the second and third case also relate similarly. As Theorem~\ref{thm:cutvertices} gives us suitable monomials, when we add back the path we previously deleted, an application of lemma \ref{le:path} finishes the proof.
\end{proof}

\begin{theorem}\label{thm:nontri2}
	Let $G$ be a 2-connected $xy$-fundamental graph with exactly one non-triangular interior face, and that face does not contain $x$ nor $y$. Let $V(G) = \lbrace x, y, a, b, c, d, v_1, \dots, v_n \rbrace$, where $a, b$ and $c,d$ are the two pairs of vertices of the non-triangular face belonging to the neighbouring interior faces with $ab \in E(G)$ and $cd \in E(G)$, and let $C(v)$ be the colour class of vertex $v$ in the 3-colouring of the subgraphs of $G$ created by deleting the paths connecting $a$ with $c$ and $b$ with $d$. Then in $P(G)$:
	\begin{enumerate}
		\item there is a non-vanishing monomial of the form $\eta_1 x^0 y^1 a^{\alpha_a} b^{\alpha_b} c^{\alpha_c} d^{\alpha_d} \Pi v_i^{\alpha_i}$, $\alpha_k \le 2$, if $d(a,c)=1$, $C(x)=C(a)$ and $C(y)=C(c)$ OR $d(b,d)=1$, $C(x)=C(b)$ and $ C(y)=C(d)$;
		\item there is a non-vanishing monomial of the form $\eta_2 x^0 y^0 a^{\alpha_a} b^{\alpha_b} c^{\alpha_c} d^{\alpha_d} \Pi v_i^{\alpha_i}$, $\alpha_k \le 2$ otherwise;
	\end{enumerate}
\end{theorem}
\begin{figure}[htb]
\begin{center}
	\includegraphics[width = 0.7\textwidth]{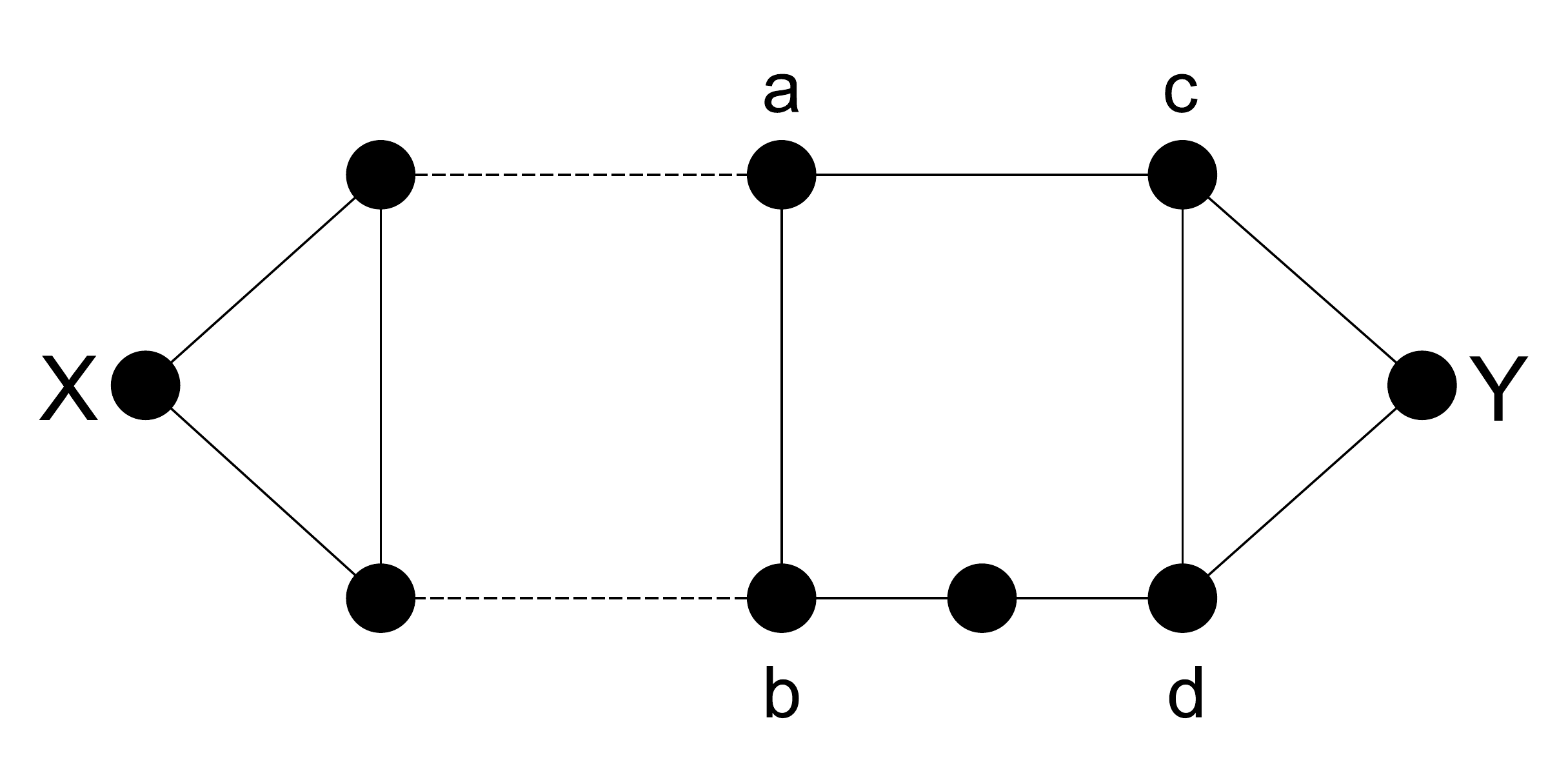}
\end{center}
\caption{An example of labelling described in Theorem~\ref{thm:nontri2}.} \label{fig:nontri2}
\end{figure}
\begin{proof}
	Suppose at first that $C(x)=C(a)$ and $C(y)=C(c)$. We can connect vertex $a$ with $d$, and if $d(b,d) > 1$, also with every interior vertex on the path connecting $b$ with $d$, thus obtaining an $xy$-fundamental 2-connected near triangulation $G'$. If $d(a,c) = 1$, then $C(a) \neq C(c)$, thus $C(x) \neq C(y)$, and by Corollary~\ref{cor:Hutch} $P(G')$ contains a non-vanishing monomial with $x^0$, $y^1$ and every other exponent equals 2. As neither $x$ nor $y$ were affected by addition of edges to $G$, $P(G)$ contains a non-vanishing monomial of the form $\eta_1 x^0 y^1 a^{\alpha_a} b^{\alpha_b} c^{\alpha_c} d^{\alpha_d} \Pi v_i^{\alpha_i}$, $\alpha_k \le 2$. If $d(a,c) > 1$, then $G'$ fulfils the conditions of Theorem~\ref{thm:nontri1}, with $d$ serving as vertex $a$ in the statement of that theorem. Moreover, as $C(x)=C(a)$ and $C(y)=C(c)$, and $d$ neighbours both $a$ and $c$ in $G'$, then in colouring of $G'$ $C(x)\neq C(d)$ and $C(y) \neq C(d)$. Hence by Theorem~\ref{thm:nontri1} $P(G')$ contains a non-vanishing monomial with $x^0$, $y^0$ and every other exponent no larger than 2, and this again implies that there is a non-vanishing monomial of the form $\eta_2 x^0 y^0 a^{\alpha_a} b^{\alpha_b} c^{\alpha_c} d^{\alpha_d} \Pi v_i^{\alpha_i}$, $\alpha_k \le 2$ in $P(G)$. The case when $C(x)=C(b)$ and $ C(y)=C(d)$ is analogous.
	
	Suppose now that $C(x) \neq C(a)$ and $C(y)=C(c)$. Start by removing the paths from $a$ to $c$ and $b$ to $d$ from $G$. This leaves us with two separate, 2-connected near triangulations $G'$ and $G''$ with $\lbrace x,a,b \rbrace \in V(G')$ and $\lbrace y,c,d \rbrace \in V(G'')$. As $C(y)=C(c)$, then $C(y) \neq C(d)$, and by Corollary~\ref{cor:Hutch} in $P(G'')$ there is a non-vanishing monomial of the form $\eta_1 y^0d^1c^2 \Pi v_i^2$. Now as $C(x) \neq C(a)$, there exists a non-vanishing monomial $\eta_1 x^0a^1b^2 \Pi u_i^2$ in $P(G')$, as the polynomial of $xa$-fundamental subgraph of $G'$ contains a non-vanishing monomial with $x^0$ and $a^1$, and as $G'$ is a 2-connected near triangulation, every other exponent must be equal to 2. Now add back the previously removed paths. Each of them contains in its graph polynomial a non-vanishing monomial with every exponent equal to 1, except for one of its endpoints, which has power 0. We will call that monomial \emph{oriented towards} the endpoint with non-zero exponent. Add paths connecting $a$ with $c$ and $b$ with $d$ to $G'$ and $G''$, and by multiplication of the monomials described above we obtain a monomial of the form $\eta_2 x^0 y^0 a^{\alpha_a} b^{\alpha_b} c^{\alpha_c} d^{\alpha_d} \Pi v_i^{\alpha_i}$, $\alpha_k \le 2$ in $P(G)$, where exponent of each of the vertices $a,b,c,d$ is equal to 2. This monomial does not vanish, as the only other way to get this monomial would require us to orient both of the paths in the opposite direction, but this would imply that there were a non-vanishing monomial  $\eta_1 y^0d^2c^1 \Pi v_i^2$ in $P(G'')$, which is not the case as $C(y)=C(c)$. Cases where $C(x) = C(a)$ and $C(y) \neq C(c)$, $C(x) \neq C(b)$ and $C(y)=C(d)$ or $C(x) = C(b)$ and $C(y) \neq C(d)$ are sorted out in the same manner.
	
	The last case is when $C(x) \neq C(a)$ and $C(y) \neq C(c)$. Observe at first, that we can also assume that $C(x) \neq C(b)$ and $C(y) \neq C(d)$, as all the other cases were already solved in previous arguments due to symmetries. Let $G'$ and $G''$ be as in previous case. As $C(b) \neq C(x) \neq C(a)$, then in $P(G')$ there are non-vanishing monomials $\eta_1 x^0a^1b^2 \Pi v_i^2$ and $-\eta_1 x^0a^2b^1 \Pi v_i^2$. Similarly, there are non-vanishing monomials $\eta_2 y^0c^1d^2 \Pi u_i^2$ and $-\eta_2 y^0c^2d^1 \Pi u_i^2$ in $P(G'')$. Now reconstruct $G$ as previously, orienting path connecting $a$ and $c$ towards $a$ and path connecting $b$ and $d$ towards $d$. To comply with requirements of the thesis, we have to use the first and fourth monomial from those specified above, thus in $P(G)$ we have a monomial $-\eta_1 \eta_2 x^0y^0a^2b^2c^2d^2 \Pi v_i^2$. The only other way to reach this set of exponents is to use the second and third monomial, and orient paths in opposite directions, but as a simultaneous switch of orientations preserves sign, we again obtain $-\eta_1 \eta_2 x^0y^0a^2b^2c^2d^2 \Pi v_i^2$, so those monomials do not annihilate each other, but rather double the coefficient. As all cases are now addressed, the proof is complete.
\end{proof}
\vspace{1em}

\subsection{General outerplanar graphs}~\\

The three theorems above can be combined with Theorem~\ref{thm:cutvertices} to obtain a general characterisation of $(i,j)$-extendability of outerplanar graphs. We will start with some technicalities.

\begin{definition}
	Let $G$ be an outerplanar graph. A non-triangular inner face of $G$ will be called \emph{type 0} if it is as defined in Theorem~\ref{thm:nontri0} (with possibly $y$ belonging to that face instead of $x$), \emph{type 1} if it is as defined in Theorem~\ref{thm:nontri1} and \emph{type 2} if it is as defined in Theorem~\ref{thm:nontri2}. In case of type 1 faces, the vertex belonging to the two neighbouring faces will be called an \emph{apex} of that face.
\end{definition}

\begin{lemma}\label{lem:2conn}
	Let $G$ be a connected outerplanar graph with $V(G)=\lbrace x,y,v_1, \dots, v_i \rbrace$ and let $G'$ be a supergraph of $G$ obtained by adding a path of the length 2 to $G$ in a way that preserves outerplanarity. Then the monomial $x^{\alpha_x}y^{\alpha_y}\Pi v_i^{\alpha_i}$ does not vanish in $P(G)$ if and only if the monomial $x^{\alpha_x}y^{\alpha_y}\Pi v_i^{\alpha_i}z^2$ does not vanish in $P(G')$, where $z$ is the middle vertex of the added path.
\end{lemma}
\begin{proof}
	The implication from $P(G)$ to $P(G')$ is obvious and was shown to be true and utilized multiple times in this paper. Suppose there is a non-vanishing monomial $x^{\alpha_x}y^{\alpha_y}\Pi v_i^{\alpha_i}z^2$ in $P(G')$. As $P(G')=P(G)(ab-az-bz+z^2)$, where $a,b$ are the endpoints of the added path, and none of the monomials from $P(G)$ contains $z$ due to the fact that $z \notin V(G)$, then the only way to obtain the monomial above is by multiplying $x^{\alpha_x}y^{\alpha_y}\Pi v_i^{\alpha_i}$ by $z^2$, thus the former must occur in $P(G)$.
\end{proof}

\begin{definition}
	Let $G$ be a 1-connected fundamental outerplanar graph. For every cutvertex of $G$ that is not an endpoint of any bridge add a path of length 2, connecting the pair of some neighbours of that cutvertex without disrupting outerplanarity, thus creating a non-triangular face of type 0. Then for every bridge or chain of bridges of $G$ add a path of the length 2 connected to the pair of the neighbours of the endpoints of that bridge or chain of bridges (or to the neighbour and the endpoint if it has degree 1) in a way that preserves outerplanarity, creating a face of type 2 (or type 0). Finally, if $G$ is a path, connect endpoints of that path with a path of length 2. The resulting supergraph of $G$ will be called a \emph{2-connection} of $G$. The 2-connection of A 2-connected graph would be the graph itself.
\end{definition}
Notice, that the 2-connection of a 1-connected graph is not unique --- for example, the graph on Figure~\ref{fig:2conect} has 8 different 2-connections. However, each of the 2-connections has the same relevant properties --- namely the color classes of the cutvertices and types of the newly created non-triangular faces.

\begin{figure}[htb]
\begin{center}
	\includegraphics[width = 0.9\textwidth]{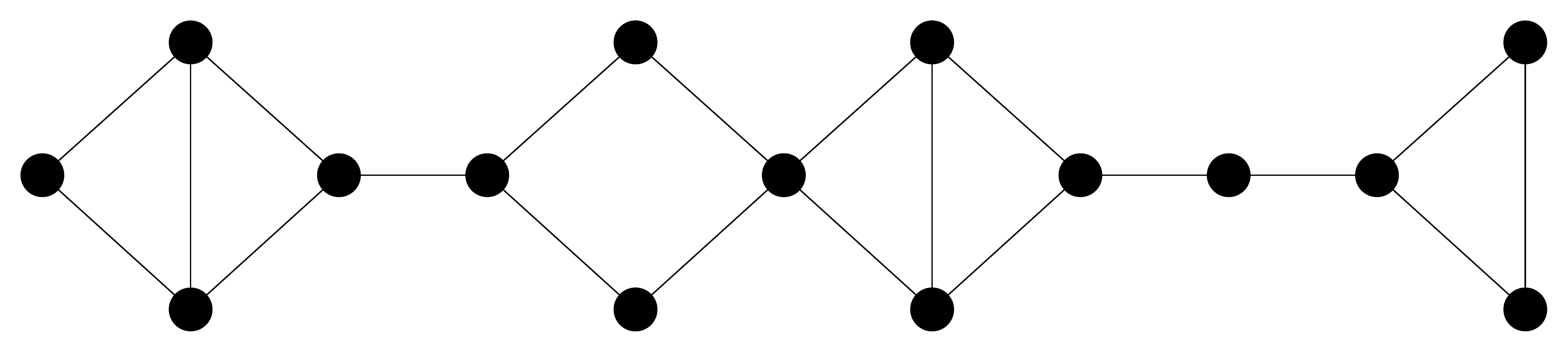}
\end{center}
\begin{center}
	\includegraphics[width = 0.9\textwidth]{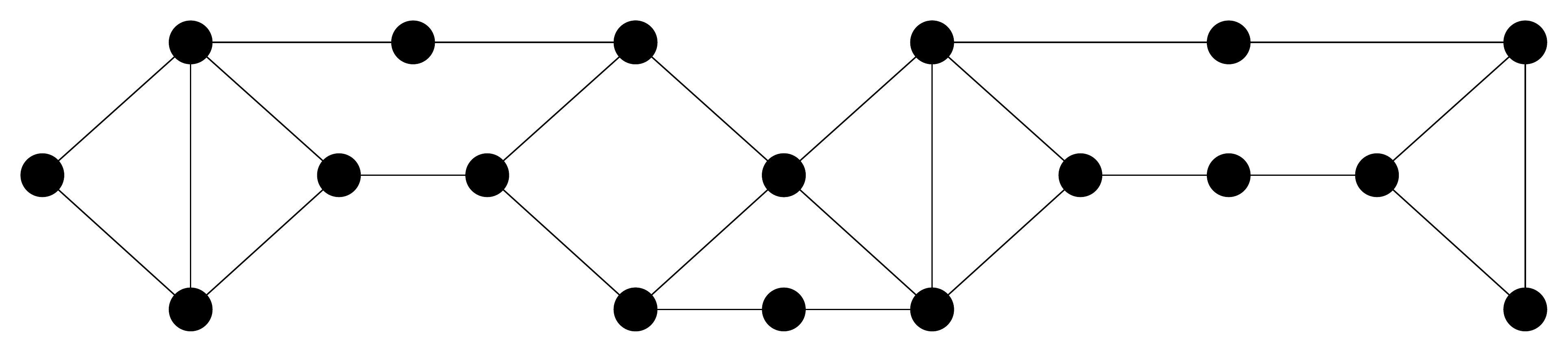}
\end{center}
\caption{Top: a connected, outerplanar graph $G$; Bottom: a possible 2-connection of $G$.} \label{fig:2conect}
\end{figure}

The following remark is a direct consequence of lemma~\ref{lem:2conn}.

\begin{remark}
	Let $G$ be a connected $xy$-fundamental outerplanar graph, $V(G)=\lbrace x,y,v_1, \dots, v_m \rbrace$ and let $G'$ be its 2-connection, $V(G')=\lbrace x,y,v_1, \dots, v_m, u_1, \dots, u_n \rbrace$. There is a non-vanishing monomial $x^{\alpha_x}y^{\alpha_y}\Pi v_i^{\alpha_i}$ in $P(G)$ if and only if there is a non-vanishing monomial $x^{\alpha_x}y^{\alpha_y}\Pi v_i^{\alpha_i} \Pi u_j^2$ in $P(G')$.
\end{remark}

The following theorem presents a full characterisation of the polynomial extendability of connected fundamental outerplanar graphs.

\begin{theorem} \label{thm:general}
	Let $G$ be a connected $xy$-fundamental outerplanar graph, $V(G)=\lbrace x,y,v_1, \dots, v_i \rbrace$, and let $G'$ be a 2-connection of $G$. Then in $P(G)$:
	\begin{enumerate}
		\item there is a non-vanishing monomial of the form $x^1y^1\Pi v_i^{\alpha_i}$, $\alpha_k \le 2$ if $G$ is a 2-connected near-triangulation with $C(x)=C(y)$ OR $G$ is as in point 1 of Theorem~\ref{thm:cutvertices} OR every non-triangular face of $G'$ is of type 1 and every apex, $x$ and $y$ have the same colour in every 3-colouring of $G$.
		\item there is a non-vanishing monomial of the form $x^0y^1\Pi v_i^{\alpha_i}$, $\alpha_k \le 2$ if $G$ is a 2-connected near-triangulation with $C(x) \neq C(y)$ OR $G$ is as in point 2 of Theorem~\ref{thm:cutvertices} OR $G'$ is as in point 1 of Theorem~\ref{thm:nontri0} OR $G'$ is as in point 1 of Theorem~\ref{thm:nontri2} OR every non-triangular face of $G'$ is of type 1 and in every 3-colouring of $G'$ there is exactly one pair of consecutive apexes (or either $x$ or $y$ with the closest apex) with different colours OR only one of the non-triangular faces of $G'$ is not of the type 1 and conditions of point 1 of Theorem~\ref{thm:nontri2} are fulfilled on that face.
		\item there is a non-vanishing monomial of the form $x^0y^0\Pi v_i^{\alpha_i}$, $\alpha_k \le 2$ otherwise.
	\end{enumerate}	
\end{theorem}	

\begin{proof}
	We will omit every case that is covered already by previous theorems, leaving us only with the cases when there are multiple non-triangular faces. Suppose all of those are of type 1. It is easy to see (with some help of lemma~\ref{le:path}) that for every such face removal of all vertices belonging only to this (and outer) face produces a cutvertex, simultaneously changing nothing in terms of extendability-relevant monomials. Hence apply Theorem~\ref{thm:cutvertices}, with each apex acting as a cutvertex.
	
	Suppose now there is a face of type either 0 or 2 in $G'$. Theorems~\ref{thm:nontri0} and \ref{thm:nontri2} show that the only cases where there is no monomial in $P(G')$ (and thus in $P(G)$) with both $x$ and $y$ in power 0 is when 3-colouring $G'$ we cannot avoid a situation described in point 1 of either of these theorems on any of such faces, and in those cases there is a non-vanishing monomial with $x^0$ and $y^1$. Observe that this is not the case when there are at least two faces of type 0 or 2, as we can avoid this situation by either permuting the colours, or by changing them on vertices of degree 2 (as in case of type 0 faces at least one such vertex other than $x$ and $y$ definitely exists). So there are only two cases when we cannot avoid that. The first is when in $G'$ there is only one face of type 2, no faces of type 0, there is a pair of neighbouring vertices belonging to this face such that the only other face of $G'$ they belong to simultaneously is the outer face, and in any 3-colouring of $G$ (and thus also $G'$) each of those vertices has the same colour as $x$ or $y$, depending on which of those vertices lies on the same "side" of that face. Label the vertex from this pair lying closer to $x$ as $v_x$, and the one being closer to $y$ as $v_y$. The case of $C(x)=C(v_x)$ can occur either when on one side there are only triangular faces between $x$ and $v_x$, with the structure of that triangulation forcing the same colour of those vertices, or when for every type 1 face between those vertices, the triangular structure between neighbouring faces or between $x$/$v_x$ and the nearest such face forces the same colour on each of those vertices. The same is true for $y$ and $v_y$, with the restriction that the former situation cannot occur for both of those pairs. The second case is when there is exactly one face of type 0 in $G'$ (without loss of generality we can assume that $x$ lies on that face), no faces of type 2, $x$ has a neighbour ($v_0$) that lies also on adjacent inner face, and the colour of that vertex is the same as colour of $y$ in every 3-colouring of $G'$. This can be only caused by the fact that the apex of every type 1 face is forced to have the same colour as the others, as well as $y$ and $v_0$.
\end{proof}

Finally, we prove that Theorem~\ref{thm:general} can be restated as Theorem~\ref{thm:general2}.

\begin{proof}[Proof of Theorem~\ref{thm:general2}]
Neither the graph polynomial nor the colouring depends on a particular graph embedding. Therefore, let $G$ be any outerplanar graph with $V(G) = \{x, y, v_1, \dots, v_n\}$. At first notice, that if $G$ is not connected and $x$ and $y$ are in different connected components, one may use Theorem~\ref{thm:zhu} directly to obtain a monomial with $\beta = \gamma = 0$, so then obviously the third case occurs.

For $x$ and $y$ in one component observe that by the Hutchinson's shrinking argument it is enough to prove theorem for $G$ being $xy$-fundamental graph. See the proof of Corollary~\ref{cor:Hutch} for details. Now consider consequences of each of the situations described in the statement of Theorem~\ref{thm:general} in terms of 3-colourings. In every case of point (i) we obviously have that $C(x)=C(y)$. Moving to the second point, the first condition again directly states that $C(x) \neq C(y)$. If $G$ is as in point 2 of Theorem~\ref{thm:cutvertices} or every non-triangular face of $G'$ is of type 1 and in every 3-colouring of $G'$ there is exactly one pair of consecutive apexes (or either $x$ or $y$ with the closest apex) with different colours, as the colour class changes only once on the cutvertices/apexes, then obviously classes of terminal vertices $x$ and $y$ have to be different. If $G'$ is as in point 1 of Theorem~\ref{thm:nontri0}, then it is directly stated that the colour of one of terminal vertices is the same as the colour of one of the neighbours of the other terminal vertex, thus the colours of terminal vertices have to be different. Finally, if $G'$ is as in point 1 of Theorem~\ref{thm:nontri2} or only one of the non-triangular faces of $G'$ is not of the type 1 and the conditions of point 1 of Theorem~\ref{thm:nontri2} are fulfilled on that face, the vertices $x$ and $y$ are in the same colour class as vertices $a$ and $c$ (or $b$ and $d$), respectively, and those vertices are adjacent, hence their colours cannot possibly be the same.

Finally, observe that in any other case the colour classes of $x$ and $y$ are independent --- the structure of the graph permits the colours to be rearranged in some parts without changing the colours in the other parts, therefore the graph can be properly 3-coloured with both $C(x)=C(y)$ and $C(x) \neq C(y)$. As an example consider point (ii) of Theorem~\ref{thm:nontri0}, other cases are analogous. Starting with the triangulated part of the graph (i.e. the graph minus internal vertices of the path between $a$ and $b$ containing $x$) already coloured, analyse possible proper 3-colourings of the path from $a$ to $b$. If $\min(d(x,a),d(x,b))>1$, then we can colour $x$ with any of the 3 colours. Otherwise, suppose without loss of generality that $d(x, a) = 1$ and hence $d(x, b) > 1$. Then $x$ can be coloured with any colour except $C(a)$, but there is $C(a) \neq C(y)$. Therefore, again $x$ can get colour of $y$ or some different one.
\end{proof}

\section{Further work}
In \cite{PoThI} and \cite{PoThIII} Postle and Thomas provided extendability results for triangulated planar graphs. Namely, their results are summarized in the following theorem.
\begin{theorem}
Let $G = (V, E)$ be any plane graph, let $C \subseteq V$ be the set of vertices on the outer face, $x, y \in C$, $x \neq y$. Then
\begin{enumerate}
\item $(G, x, y)$ is $(1,2)$-extendable if and only if there exists a proper colouring $c \colon C \to \{1,2,3\}$ such that $c(x) \neq c(y)$;
\item $(G, x, y)$ is $(2,2)$-extendable.
\end{enumerate}
\end{theorem}
Extendability is naturally transformed into plane graphs by allowing interior vertices to have a list of colours of length 5. 
One may ask, whether is it possible to restate the above theorem in the terms of a graph polynomial, i.e. to extend, at least partially Theorem~\ref{thm:general2} to planar graphs. Our partial results suggest that it is possible.

\end{document}